\documentclass[a4paper,12pt]{article}

\usepackage{amsmath,amssymb,amsthm,amscd,cases}

{}
\numberwithin{equation}{section}

\newtheorem{thm}{Theorem}[section]

\newcommand{\R}{\mathbb{R}}

\renewcommand{\a}{\alpha}
\renewcommand{\b}{\beta}

\newcommand{\p}{\partial}

\newcommand{\bc}{\begin{cor}}
\newcommand{\ec}{\end{cor}}
\newcommand{\bl}{\begin{lem}}
\newcommand{\el}{\end{lem}}
\newcommand{\bp}{\begin{prop}}
\newcommand{\ep}{\end{prop}}
\newcommand{\bt}{\begin{thm}}
\newcommand{\et}{\end{thm}}
\newcommand{\bal}{\begin{array}{ll}}
\newcommand{\ba}{\begin{array}}
\newcommand{\bac}{\begin{array}{ccc}}
\newcommand{\ea}{\end{array}}

\newcommand{\be}{\begin{equation}}
\newcommand{\ee}{\end{equation}}

\newtheorem{lem}[thm]{Lemma}
\newtheorem{cor}[thm]{Corollary}
\newtheorem{prop}[thm]{Proposition}

\usepackage{showkeys}

\begin{document}
\title {\Large\bf{The Pointwise Estimates of Solutions for Semilinear Dissipative
Wave Equation
   }}

\author { Yongqin Liu{\footnote{email:~yqliu2@yahoo.com.cn}}\\
 {\footnotesize\emph{$\begin{array}{ll}&Department of Mathematics, Fudan University, Shanghai, China\\&Department of Mathematics, Kyushu University, Fukuoka,
 Japan\end {array}$}}\\}
\date{}
\maketitle
\begin{abstract}

 In this paper we focus on the global-in-time existence and the
  pointwise  estimates of solutions to  the initial value  problem for
 the semilinear dissipative wave equation in multi-dimensions. By using the method of Green function combined
 with the energy estimates, we obtain the   pointwise  decay estimates of solutions to the problem.
 \\
\indent {\bf keywords}:~semilinear dissipative wave equation,
pointwise estimates, Green function. \\
\indent {\bf MSC(2000)}:  ~35E15; ~35L15.
\end{abstract}

\section{Introduction}
In this paper we consider the initial value problem for the semilinear dissipative wave
equation in $n(n\geq1)$ dimensions,
\begin{equation}
\label{1a} 
(\Box+\partial_t) u(x,t)=f(u),\ x\in{\mathbb{R}^n},\ t>0,
\end{equation}
with initial condition
\begin{equation}\label{IC}
(u,\partial_tu)(x,0)=(u_0,u_1)(x),\ 
x\in{\mathbb{R}^n},
\end{equation}
where $\Box+\partial_t=\partial_t^2-\triangle_x+\partial_t$ is the
dissipative wave operator with Laplacian
$\triangle_x=\sum\limits_{j=1}\limits^{n}\partial_{x_j}^2,$
$f(u)=-|u|^{\theta}u$, $\theta>0$ is an integer. Equation (\ref{1a}) is often
called the semilinear dissipative wave equation or semilinear
telegraph equation.

There have been many results on the equation (\ref{1a}) and its
variants corresponding to the different forms of $f(u)$. By
employing the weighted $L^2$  energy method and the explicit formula
of solutions, Ikehata, Nishihara and Zhao \cite{INZ} obtained that
the behavior of solutions to (\ref{1a}) as $t\rightarrow \infty$ is
expected to be same as that for the corresponding heat equation,
Nishihara \cite{Ni0} studied the global asymptotic behaviors in
three and four dimensions,  and Nishihara and Zhao \cite{NZ}
obtained the decay properties of solutions to the problem
\eqref{1a}\eqref{IC}.
 Kawashima,
Nakao and Ono \cite{KNO} studied the decay property of solutions to
(\ref{1a}) by using the energy method combined with $L^p-L^q$
estimates, and Ono \cite{Ono1} derived sharp decay rates in the
subcritical case of solutions to (\ref{1a}) in unbounded domains in
$\mathbb{R}^N$ without the smallness condition on initial data.
Also, recently Nishihara, etc.  in \cite {Ni1, Ni2} studied the following semilinear damped wave equations  with time or space-time
 dependent damping term,
\begin{equation}\label{n1}
u_{tt}-\Delta u+b(t)u_t+|u|^{\rho-1}u=0,
\end{equation}
and 
\begin{equation}\label{n2}
u_{tt}-\Delta u+b(t,x)u_t+|u|^{\rho-1}u=0,
\end{equation}
where $\rho>1,$ $b(t)=b_0(1+t)^{-\beta}$ with $b_0>0, -1<\beta<1$, and $b(t,x)=b_0(1+
|x|^2)^{-{\alpha\over2}}(1+t)^{-\beta}$ with $b_0>0, \alpha\geq0, \beta\geq0, \alpha+\beta\in [0,1)$,
  and obtained the global existence and the $L^2$ decay rate of the solution by using the weighted energy method. \eqref{n1} and \eqref{n2} with the exponents $\a=\b=0$ yield \eqref {1a}. For
 studies on the case $f(u)=|u|^{\theta}u$, see \cite{HO, IMN, IO,
Na, Ni, NO}, for studies on the case $f(u)=|u|^{\theta+1}$, see
\cite{Ik, LZ, Ono2, Ono3, TY, Z}, and for studies on the global
attractors, see \cite{BP, KS} and the references cited there.

The main purpose of this paper is to study the pointwise estimates
of  solutions for \eqref{1a}\eqref{IC}.   In \cite{LW}, Liu and Wang studied the corresponding linear problem, i.e. \eqref{1a} with $f(u)=0$ and \eqref{IC}, and obtained the pointwise estimates of solutions. In this paper, we first obtain the global-in-time
solutions by energy method combined with the fixed point theorem of
Banach, and then obtain the optimal pointwise decay estimates of
the solutions  by using the properties of the Green function proved in \cite{LW} combined with
Fourier analysis. One point worthy to be mentioned is that, different from that for solutions to the corresponding linear problem, the order of derivatives with respect to time variable $t$ of solutions does not contribute to the decay rate of solutions due to the presence of the semilinear term, which could be seen from \eqref{pe1} in Theorem \ref{pe} and \eqref{lr1} in Theorem \ref{lr}.

The rest of the paper is arranged as follows. In section 2, the main results are stated.  We give the
 proof of Proposition \ref{ae} and then obtain the global-in-time existence of solutions in section 3. In section 4 we give estimates on
 the Green function by Fourier analysis which will be used in the last section  where the proof of Theorem \ref{pe} is given.

Before closing this section, we give some notations to be used 
below. Let $\mathcal{F}[f]$ denote the Fourier transform of $f$ 
defined by 
$$
\mathcal{F}[f](\xi)=\hat{f}(\xi)
:=\int_{\R^n}e^{-ix\cdot\xi}f(x)dx,
$$
and we denote its inverse transform by $\mathcal{F}^{-1}$. 

For $1\leq p\leq\infty$, $L^p=L^p(\R^n)$ is the usual Lebesgue space 
with the norm $\|\cdot\|_{L^p}$. 
Let $s$ be a nonnegative integer. 
Then $H^s=H^s(\R^n)$ denotes the Sobolev space of $L^2$ functions, 
equipped with the norm 
$$
\|f\|_{H^s}
:=\Big(\sum\limits_{k=0}\limits^{s}
\|\partial_x^kf\|_{L^2}^2\Big)^{1\over2}.
$$
In particular, we use $\|\cdot\|=\|\cdot\|_{L^2}$, $\|\cdot\|_s=\|\cdot\|_{H^s}.$
Here, for a multi-index $\alpha$, $D_x^{\alpha}$ denotes the totality 
of all the $|\alpha|$-th order derivatives with respect to 
$x\in{\mathbb R}^n$. 
Also, $C^k(I; H^s(\R^n))$ denotes the space of $k$-times 
continuously differentiable functions on the interval $I$ 
with values in the Sobolev space $H^s=H^s(\R^n)$. 

Finally, in this paper, we denote every positive constant by the 
same symbol $C$ or $c$ without confusion. 
$[\,\cdot\,]$ is Gauss' symbol. 

\section{Main theorems}\ 
The first main result is about the global existence of solutions to the initial value problem \eqref{1a}\eqref{IC}.
\begin{thm}[Global existence]
\label{ge}  \ Let $\theta>0$ be an integer. Assume that  $(u_0, u_1)\in H^{s+1}(\R^n)\times H^{s}(\R^n), \ s\geq[{n\over2}]$ ,
put
$$
E_0:=\|u_0\|_{H^{s+1}}+\|u_1\|_{H^{s}}.
$$
Then if $E_0$ is suitably small, 
\eqref{1a} \eqref {IC} admits a unique solution 
$$
u\in
\bigcap\limits^{s+1}\limits_{i=0} C^i([0,\infty);
H^{s+1-i}(\mathbb{R}^n)) ,
$$ 
which satisfies
\begin{equation}\label{ge1}
\sum\limits_{i=0}\limits^{s+1}\|\partial_t^iu(
t)\|^2_{s+1-i}+\int^t_0(\|\nabla
u(\tau)\|^2_{s}+\sum\limits_{i=1}\limits^{s+1}\|\partial_{\tau}^iu(\tau)\|^2_{s+1-i})d\tau\leq
CE_0^2.
\end{equation}
\end{thm}
Theorem \ref{ge} is  proved by combining the local existence of solutions stated in the following Theorem \ref{le}  with a priori estimate in the following Proposition \ref{ae}.

\begin{thm} [Local existence]\label{le}\
Let $\theta>0$ be an integer.
 Assume that $(u_0, u_1)\in H^{s+1}(\R^n)\times H^{s}(\R^n), \  s\geq[{n\over2}]$,  then there exists $T>0$ and a unique solution to
\eqref{1a} \eqref{IC} satisfying  $$u\in \bigcap\limits^{s+1}\limits_{i=0}
C^i([0,T); H^{s+1-i}(\mathbb{R}^n)) .$$
\end{thm}
The proof of the local existence result is based on the fixed point theorem of Banach and standard argument, so the detail is omitted.

 Based on the  a priori assumption
\begin{equation}\label{aa}
\sup\limits_{0<t<T}\|u(t)\|_{L^{\infty}}\leq \bar{\delta},
\end{equation}
where $s>n$ is  an integer and $\bar{\delta}<1$ is a small constant,
the following a priori estimate  is obtained.
\bp[A priori estimate]\label{ae}
Under the same assumptions as in Theorem \ref{ge}, 
let $u(x,t)$ be the solution to \eqref{1a}\eqref{IC} which is defined on $[0,T]$ and verifies \eqref{aa}, then
the following estimate holds,
\begin{equation}\label{ae1}
\sup\limits_{0<t<T}\{\sum\limits_{i=0}\limits^{s+1}
\|\partial_t^iu(
t)\|^2_{s+1-i}\}+\int^T_0(\|\nabla
u(\tau)\|^2_{s}+\sum\limits_{i=1}\limits^{s+1}\|\partial_{\tau}^iu(\tau)\|^2_{s+1-i})d\tau\leq
CE_0^2.
\end{equation}                                        
\ep
{\bf Remark 1.}
In (\ref{1a}),  $f(u)=-|u|^{\theta}u$ is called absorption term which makes it possible to close energy estimates. Otherwise, if $f(u)=|u|^{\theta}u$, then Theorem \ref{ge} does
not hold, since  the lower-order term present in the 
energy estimates could not be controlled.

The second main result is about the pointwise estimate to the solution obtained in Theorem \ref{ge}.
\begin{thm}[Pointwise estimate]\label{pe}
Under the same assumptions as in Theorem \ref{ge}, if  $s>n,$ $\theta\geq 2+[{1\over n}]$, and for any
multi-indexes $\alpha$, $|\alpha|<s-{n\over2}$, there exists some
constant $r>\max{\{{n\over2}, 1\}}$ such that
$$
|D^{\alpha}_xu_0(x)|+|D^{\alpha}_xu_1(x)|\leq C(1+|x|^2)^{-r},
$$
then for $h\geq0$ satisfying $|\alpha|+h<s-n,$ the
solution  to \eqref{1a}\eqref{IC} obtained in Theorem \ref{ge} satisfies the following pointwise estimate,
\begin{equation}\label{pe1}
|\partial_t^hD_x^{\alpha}u(x,t)|\leq CE_0
(1+t)^{-{{n+|\alpha|}\over2}}(1+{{|x|^2}\over{1+t}})^{-r}.
\end{equation}
\end{thm}
{\bf Remark 2.} \ From the estimate in Theorem \ref{pe}, we see that the order of derivatives with respect to $t$ of the solution obtained in Theorem \ref{ge}  has no effect on the decay rate of  the solution, as is different from that for the solution to the corresponding linear problem  studied in \cite{LW}.

As a direct corollary of Theorem \ref{pe} we have
\bc\label{44}
Assume that the same assumptions as in Theorem \ref{pe} hold, then
for $p\in[1,\infty]$, $|\alpha|+h<s-n$, the solution to \eqref{1a}\eqref{IC}
satisfies,
$$
\|\partial_t^hD_x^{\alpha}u(\cdot,t)\|_{L^p}\leq CE_0
(1+t)^{-{n\over2}(1-{1\over p})-{{|\alpha|}\over2}}.
$$
\ec

\section{The global existence of solutions}

First we give a lemma which will be used in our next energy estimates.
\bl\label{L}
Let $n\geq 1$, $1\leq p,q,r\leq\infty$ and 
${1\over p}={1\over q}+{1\over r}$. Then the following estimate holds: 
\be\label{d1} 
\|\p^k_x(uv)\|_{L^p}\leq
C(\|u\|_{L^q}\|\p^k_xv\|_{L^r}+\|v\|_{L^q}\|\p^k_xu\|_{L^r})
\ee
for $k\geq 0$.
\el

\begin{proof}
The estimate \eqref{d1} can be found in a literature but we give here a proof. 
To prove \eqref{d1}, it is enough to show that, for $k_1\geq1$, 
$k_2\geq1$ and $k_1+k_2=k$, the following estimate holds: 
$$
\|\p^{k_1}_xu\,\p^{k_2}_xv\|_{L^p}\leq
C(\|u\|_{L^q}\|\p^k_xv\|_{L^r}+\|v\|_{L^q}\|\p^k_xu\|_{L^r}).
$$
Let $\theta_j={k_j\over k}$, $j=1,2$, and define $p_j$, $j=1,2$, by 
$$
{1\over p_j}-{k_j\over n}
=(1-\theta_j){1\over q}+\theta_j({1\over r}-{k\over n}).
$$ 
Since $\theta_1+\theta_2=1$, we have 
${1\over p}={1\over p_1}+{1\over p_2}$. 
By using the H\"{o}lder inequality and the Gagliardo-Nirenberg 
inequality, we have 
\begin{equation*}
\begin{split}
\|\p^{k_1}_xu\,\p^{k_2}_xv\|_{L^p}\leq &\ 
\|\p^{k_1}_xu\|_{L^{p_1}}\|\p^{k_2}_xv\|_{L^{p_2}}\\[1mm]
\leq &\ 
C(\|u\|^{1-\theta_1}_{L^{q}}\|\p^{k}_xu\|^{\theta_1}_{L^{r}})
(\|v\|^{1-\theta_2}_{L^{q}}\|\p^{k}_xv\|^{\theta_2}_{L^{r}})\\[1mm]
\leq &\ 
C(\|u\|_{L^{q}}\|\p^{k}_xv\|_{L^{r}})^{\theta_2}
(\|v\|_{L^{q}}\|\p^{k}_xu\|_{L^{r}})^{\theta_1}\\[1mm]
\leq &\ 
C(\|u\|_{L^q}\|\p^k_xv\|_{L^r}+\|v\|_{L^q}\|\p^k_xu\|_{L^r}).
\end{split}
\end{equation*}
In the last inequality, we have used the Young inequality. Thus 
\eqref{d1} is proved. 
\end{proof}

Now, let $T>0$ and consider solutions to the problem \eqref{1a}, 
\eqref{IC}, which are defined on the time interval $[0,T]$ and 
verify the regularity mentioned in Proposition \ref{ae}.
 we derive energy estimates under the  a priori assumption \eqref{aa}.

By multiplying $(\ref{1a})$ with $u_t$ and integrating on
$\mathbb{R}^n\times (0, t)$ with respect to $(x, t)$, we get
\begin{equation}\label{3a0}
\|u_t(t)\|^2+\|\nabla
u(t)\|^2+\int_{\mathbb{R}^n}|u|^{\theta+2}(x,t)dx+\int^t_0\|u_{\tau}(\tau)\|^2d\tau
\leq
CE_0^2.
\end{equation}

$\forall \alpha$, $1\leq|\alpha|\leq s$, by multiplying
$D^{\alpha}_x(\ref{1a})$ with $D^{\alpha}_xu_t$ and integrating on
$\mathbb{R}^n\times (0, t)$ with respect to $(x, t)$,
in view of Lemma \ref{L}  and \eqref{aa}\ we get
$$\begin{array}{ll}
&\|D^{\alpha}_xu_t(t)\|^2+\|D^{\alpha}_x\nabla
u(t)\|^2+\int^t_0\|D^{\alpha}_xu_{\tau}(\tau)\|^2d\tau
\\&\\& \leq
CE_0^2+C\int^t_0\|u(\tau)\|_{L^{\infty}}^{2\theta}\|D_x^{\alpha}
u(\tau)\|^2d\tau
\\&\\& \leq
CE_0^2+C\bar{\delta}\int^t_0\|D_x^{\alpha}
u(\tau)\|^2d\tau.
\end{array}
$$

By taking sum for $\alpha$ with $1\leq |\alpha|\leq s$, it yields that
\begin{equation}\label{3a}
\begin{array}{ll}
&\|\nabla u_t(t)\|^2_{s-1}+\|\nabla^2
u(t)\|^2_{s-1}+\int^t_0\|\nabla
u_{\tau}(\tau)\|^2_{s-1}d\tau
\\&\\&\leq
CE_0^2+C\bar{\delta}\int^t_0\|\nabla
u(\tau)\|^2_{s-1}d\tau.
\end{array}
\end{equation}

 By multiplying
$(\ref{1a})$ with $u$ and integrating on $\mathbb{R}^n\times (0,
t)$ with respect to $(x, t)$,  by virtue of \eqref{3a0} we get

\begin{equation}\label{3b}
\begin{array}{ll} &\|u(t)\|^2+\int^t_0\|\nabla
u(\tau)\|^2d\tau+\int^t_0\int_{\mathbb{R}^n}(|u|^{\theta+2})(x,
\tau)dxd\tau\\&\\& \leq
C(E_0^2+\|u_t(t)\|^2+\int^t_0\|u_{\tau}(\tau)\|^2d\tau) \leq CE_0^2.
\end{array}
\end{equation}

$\forall \alpha$, $1\leq|\alpha|\leq s$, by multiplying
$D^{\alpha}_x(\ref{1a})$ with $D^{\alpha}_xu$ and integrating on
$\mathbb{R}^n\times (0, t)$ with respect to $(x, t)$,
by virtue of Lemma \ref{L} we get
$$\begin{array}{ll}
\|D^{\alpha}_xu(t)\|^2+\int^t_0\|D^{\alpha}_x\nabla
u(\tau)\|^2d\tau&\leq
C(E_0^2+\|D^{\alpha}_xu_t(t)\|^2
+\int^t_0\|D^{\alpha}_xu_{\tau}(\tau)\|^2d\tau)
\\&\\&\quad
+C\int_0^t\|u(\tau)\|_{L^{\infty}}^{2\theta}\|D_x^{\alpha} u(\tau)\|^2d\tau.
\end{array}
$$

By taking sum for $\alpha$ with $1\leq |\alpha|\leq s$ and in view of \eqref{3a} and \eqref{aa}, it yields that
\begin{equation}\label{3c}
\|\nabla u(t)\|^2_{s-1}+\int_0^t\|\nabla^2
u(\tau)\|^2_{s-1}d\tau\leq
CE_0^2+C\bar{\delta}\int^t_0\|\nabla u(\tau)\|^2_{s-1}d\tau.
\end{equation}

(\ref{3a0}), (\ref{3a}), (\ref{3b}) and (\ref{3c}) yield that

\begin{equation}\label{3d}
\|u(t)\|^2_{s+1}+\|u_t(t)\|^2_{s}+\int^t_0(\|\nabla
u(\tau)\|^2_{s}+\|u_{\tau}(\tau)\|^2_{s})d\tau\leq
CE_0^2.
\end{equation}
\begin{proof}[Proof of Proposition \ref{ae}]

To prove \eqref{ae1} in Proposition \ref{ae}, it is enough to prove that the following estimate holds for 
$\forall h\in[1,s+1],\  \forall t\in [0,T]$.

\begin{equation}\label{3e}
\sum\limits_{i=0}\limits^h\|\partial_t^iu(t)\|^2_{s+1-i}+\int^t_0(\|\nabla
u(\tau)\|^2_{s}+\sum\limits_{i=1}\limits^h\|\partial_{\tau}^iu(\tau)\|^2_{s+1-i})d\tau\leq
CE_0^2. \end{equation}

It is obvious that (\ref{3e}) holds with $h=1$ by virtue of  (\ref{3d}).
Assume that \eqref{3e} holds with $h=j(1\leq j\leq s)$, next we will prove that \eqref{3e} holds with $h=j+1.$

 From (\ref{1a}), by using induction argument  we could prove that the following two equalities hold for $k\geq1$,
\begin{equation}\label{even}
\begin{array}{ll}
\partial_t^{2k}u(x,t)=&a_{2k}\Delta^ku(x,t)+b_{2k}\Delta^{k-1}u_t(x,t)\\&+P\{\Delta^iu(x,t),\Delta^ju_t(x,t),0\leq i\leq k-1,
0\leq j\leq k-2\},\end{array}
\end{equation}
\begin{equation}\label{odd}
\begin{array}{ll}
\partial_t^{2k+1}u(x,t)=&a_{2k+1}\Delta^ku(x,t)+b_{2k+1}\Delta^{k}u_t(x,t)\\&+P\{\Delta^iu(x,t),\Delta^ju_t(x,t),0\leq i\leq k-1,
0\leq j\leq k-1\},\end{array}
\end{equation}
where $a_{2k}, \ a_{2k+1},\ b_{2k}, \ b_{2k+1}$ are constants,
$P\{\Delta^iu(x,t),\Delta^ju(x,t),0\leq i\leq k-1, 0\leq j\leq
k-2\}$ is a polynomial with arguments
$\Delta^iu(x,t),\Delta^ju(x,t),0\leq i\leq k-1, 0\leq j\leq k-2$.

Let $t=0$ in \eqref{even}\eqref{odd}, we have for $k\geq1$,
\begin{equation}\label{3h}\begin{array}{ll}
\partial_t^{2k}u(x,0)=&a_{2k}\Delta^ku_0(x)+b_{2k}\Delta^{k-1}u_1(x)\\&+P\{\Delta^iu_0(x),\Delta^ju_1(x),0\leq
i\leq k-1, 0\leq j\leq k-2\},\end{array}
\end{equation}
\begin{equation}\label{3i}\begin{array}{ll}
\partial_t^{2k+1}u(x,0)=&a_{2k+1}\Delta^ku_0(x)+b_{2k+1}\Delta^{k}u_1(x)\\&+P\{\Delta^iu_0(x),\Delta^ju_1(x),0\leq i\leq k-1,
0\leq j\leq k-1\}.\end{array}
\end{equation}

 $\forall \alpha$, $|\alpha|\leq s-j$, by multiplying
$D^{\alpha}_x\partial_t^j(\ref{1a})$ with
$D^{\alpha}_x\partial_t^{j+1}u$ and integrating on
$\mathbb{R}^n\times (0, t)$ with respect to $(x, t)$, in view of
(\ref{3e}) with $h=j$, (\ref{3h}) and (\ref{3i}) we get
$$\begin{array}{ll}
&\|D^{\alpha}_x\partial_t^{j+1}u(t)\|^2+\int^t_0\int_{\mathbb{R}^n}
\|D^{\alpha}_x\partial_{\tau}^{j+1}u(\tau)\|^2 d\tau\leq
CE_0^2.
\end{array}
$$

By taking sum for $\alpha$ with $0\leq |\alpha|\leq s-j$, it yields that
\begin{equation}\label{3f}
\|\partial_t^{j+1}u(t)\|^2_{s-j}+\int^t_0\|\partial_{\tau}^{j+1}u(\tau)\|^2_{s-j}d\tau\leq
CE_0^2. 
\end{equation}
(\ref{3e}) with $h=j$  and (\ref{3f}) yield that
\begin{equation}\label{3g}
\sum\limits_{i=0}\limits^{j+1}\|\partial_t^iu(t)\|^2_{s+1-i}+\int^t_0(\|\nabla
u(\tau)\|^2_{s}+\sum\limits_{i=1}\limits^{j+1}\|\partial_{\tau}^iu(\tau)\|^2_{s+1-i})d\tau\leq
CE_0^2.
\end{equation}
It means that \eqref{3e} holds with $h=j+1.$
Thus by induction method, we complete the proof of Proposition \ref{ae}. 
\end{proof}
Now we give the proof of Theorem \ref{ge}.
\begin{proof}[Proof of Theorem \ref{ge}]
By virtue of the a priori estimate \eqref{ae1} in Proposition \ref{ae}, we can continue a unique solution obtained in Theorem \ref{le} globally in time, provided that $E_0$ is 
suitably small, say, $E_0<\delta_0$, $\delta_0$ depends only on $\bar{\delta}$ in \eqref{aa}. The global solution thus obtained satisfies \eqref{ge1} 
and \eqref{aa}. This finishes the proof of Theorem \ref{ge}.
 \end{proof}

\section{Estimates on Green function}
In this section, 
we list some formulas and properties of the Green function obtained in \cite{LW} to make preparation for the next section about the pointwise estimates of 
solutions.

 The  Green function or the fundamental
 solution to the corresponding linear dissipative wave equation (i.e. $f(u)=0$ in (\ref{1a})) to (\ref{1a}) satisfies
$$
\left\{\begin{array}{ll} (\Box+\partial_t)G(x,t)=0,\ &x\in{\mathbb{R}}^n,\ t>0,\\&\\
G(x,0)=0,\ &x\in{\mathbb{R}}^n,\\&\\
\partial_t G(x,0)=\delta(x),\ &x\in{\mathbb{R}}^n.
\end{array}
\right.
$$
 By Fourier transform we get that,
$$
 \left\{\begin{array}{ll}
  (\partial_t^2+\partial_t)\hat{G}(\xi,t)+|\xi|^2\hat{G}(\xi,t)=0,\ &\xi\in{\mathbb{R}}^n,~t>0,\\&\\
\hat{G}(\xi,0)=0,\ &\xi\in{\mathbb{R}}^n,\\&\\
\partial_t \hat{G}(\xi,0)=1,\ &\xi\in{\mathbb{R}}^n.
\end{array}
\right.
$$
The symbol of the operator for equation (\ref{1a}) is
\begin{equation}
\label{2a}\sigma(\Box+\partial_t)= \tau^2+\tau+|\xi|^2,
\end{equation}
$\tau$ and $\xi$ correspond to ${\partial\over{\partial t}}, $ and
${1\over{\sqrt{-1}}}{\partial\over{\partial x_j}},~j=1,2,\cdots,n.$
 It is easy to see that the eigenvalues of
(\ref{2a}) are
$\tau=\mu_{\pm}(\xi)={1\over2}(-1\pm\sqrt{1-4|\xi|^2}).$ By direct
calculation we have that
$$
\hat{G}(\xi,t)=(1-4|\xi|^2)^{-{1\over2}}(e^{\mu_+(\xi)t}-e^{\mu_-(\xi)t}).
$$
For convenience we decompose
$\hat{G}(\xi,t)=\hat{G}^+(\xi,t)+\hat{G}^-(\xi,t),$ where
$$\hat{G}^{\pm}(\xi,t)=\pm\mu_0^{-1}e^{\mu_{\pm}(\xi,t)},~~\mu_0(\xi)=(1-4|\xi|^2)^{1\over2}.$$
Let
$$
\chi_1(\xi)=\left\{ \begin{array}{ll} 1,\
&|\xi|<\varepsilon,\\&\\0,\ &|\xi|>2\varepsilon,
\end{array}\right.
\  \chi_3(\xi)=\left\{ \begin{array}{ll} 1,\ &|\xi|>R,\\&\\0,\
&|\xi|<R-1,
\end{array}\right.
$$
be the smooth cut-off functions, where $\varepsilon$ and $R$ are any
fixed
positive numbers satisfying $2\varepsilon<R-1.$\\\\
Set$$\chi_2(\xi)=1-\chi_1(\xi)-\chi_3(\xi),$$ and
$$\hat{G}^{\pm}_i(\xi,t)=\chi_i(\xi)\hat{G}^{\pm}(\xi,t),~~i=1,2,3.$$
We are going to study $G^{\pm}_i(x,t)$, which is the inverse Fourier
transform corresponding to $\hat{G}^{\pm}_i(\xi,t).$

 Denote $B_N(|x|,t)=(1+{{|x|^2}\over{1+t}})^{-N}$. First we give
two propositions regarding to $G_1(x,t)$ and $ G_2(x,t)$, the proof
can be seen in \cite{LW, WY}.

\bp
\label{21} For sufficiently small $\varepsilon$, there exists
constant $C>0$, and $N>n$ such that
$$
|\partial_t^hD_x^{\alpha}G_1(x,t)|\leq
C_Nt^{-(n+|\alpha|+2h)/2}B_N(|x|,t).
$$
\ep

\bp
\label{22}
 For fixed $\varepsilon$ and $R$, there exist positive numbers $m$
, $C$ and $N>n$ such that
$$
|\partial_t^hD_x^{\alpha}G_2(x,t)|\leq Ce^{-{t\over{2m}}}B_N(|x|,t).
$$
\ep

  Next we will come to consider $G_3(x,t).$
Now we list some lemmas which are useful in dealing with the higher
frequency part.

 \bl
\label{23}
 If ${\rm supp}\ \hat{f}(\xi)\subset O_R:=\{\xi;~|\xi|>R\},$ and
 $\hat{f}(\xi)$ satisfies
 $$|\hat{f}(\xi)|\leq C,~
|D_{\xi}^{\beta}\hat{f}(\xi)|\leq C|\xi|^{-1-|\beta|},~|\beta|\geq1,
 $$
 then there exist distributions $f_1(x),~f_2(x),$ and constant $C_0$
 such that $$
f(x)=f_1(x)+f_2(x)+C_0\delta(x),
 $$
 where $\delta(x)$ is the Dirac function. Furthermore, for positive
 integer $2N>n+|\alpha|,$
 $$
|D_{x}^{\alpha}f_1(x)|\leq C(1+|x|^2)^{-N},
 $$
$$
\|f_2\|_{L^1}\leq C,~~~{\rm{supp}}~ f_2(x)\subset
\{x;~|x|<2\varepsilon_1\},
$$
with $\varepsilon_1$ being sufficiently small.
 \el

  \bl\label{24}
For any $N>0,$ $\tau\geq0,$  we have that

$$
\int_{|z|=1}{{(1+{{|x+tz|^2}\over{1+\tau}})^{-N}}}dS_z\leq
C(1+t)^{2N}(1+{{|x|^2}\over{1+\tau}})^{-N}.
$$
$$
\int_{|z|\leq1}{{(1+{{|x+tz|^2}\over{1+\tau}})^{-N}}\over{\sqrt{1+\tau}}}dV_z\leq
C(1+t)^{2N}(1+{{|x|^2}\over{1+\tau}})^{-N}.
$$
  \el

The proof of Lemma \ref{23} and Lemma \ref{24} can be seen in \cite
{WY}.\\\\
The following Kirchhoff formulas can be seen in \cite {E,HZ}.

 \bl\label{25}
Assume that $w(x,t)$ is the fundamental solution of the following
wave equation with $c=1,$
$$
\left\{\begin{array}{ll} &w_{tt}-c^2\triangle
w=0,\\&\\&w|_{t=0}=0,\\&\\&\partial_t w|_{t=0}=\delta(x).
\end{array}\right.
$$
There are constants $a_{\alpha},~b_{\alpha}$ depending only on the
spatial dimension $n\geq1$ such that, if $h\in
C^{\infty}({\mathbb{R}}^n)$, then
$$
(w\ast
h)(x,t)=\sum\limits_{0\leq|\alpha|\leq{{n-3}\over2}}a_{\alpha}t^{|\alpha|+1}\int_{|z|=1}D^{\alpha}h(x+tz)z^{\alpha}dS_z,
$$
$$
(w_t\ast
h)(x,t)=\sum\limits_{0\leq|\alpha|\leq{{n-1}\over2}}b_{\alpha}t^{|\alpha|}\int_{|z|=1}D^{\alpha}h(x+tz)z^{\alpha}dS_z,
$$
for odd $n,$ and
$$
(w\ast
h)(x,t)=\sum\limits_{0\leq|\alpha|\leq{{n-2}\over2}}a_{\alpha}t^{|\alpha|+1}\int_{|z|\leq1}{{D^{\alpha}h(x+tz)z^{\alpha}}
\over{\sqrt{1-|z|^2}}}dz,
$$
$$
(w_t\ast
h)(x,t)=\sum\limits_{0\leq|\alpha|\leq{{n}\over2}}b_{\alpha}t^{|\alpha|}\int_{|z|\leq1}{{D^{\alpha}h(x+tz)z^{\alpha}}
\over{\sqrt{1-|z|^2}}}dz,
$$
for even $n.$ Here $dS_z$ denotes surface measure on the unit sphere
in ${\mathbb{R}}^n.$
 \el

By denoting $\lambda_{\eta}=\sqrt{\eta-4}$ and then taking the
Taylor expansion for $\lambda_{\eta}$ in $\eta$, we have that
$$
\lambda_{\eta}=2\sqrt{-1}+\sum\limits_{j=1}\limits^{m-1}a_j\eta^j+O(\eta^m).
$$
Since $$
\mu_{\pm}(\xi)={{-1\pm\sqrt{1-4|\xi|^2}}\over{2}}={1\over2}(-1\pm|\xi|\sqrt{|\xi|^{-2}-4}),
$$
we have that, when $\xi$ is sufficiently large,
$$
\mu_{\pm}(\xi)={1\over2}(-1\pm2\sqrt{-1}|\xi|\pm(\sum\limits_{j=1}\limits^{m-1}a_j|\xi|^{1-2j}))+O(|\xi|^{1-2m}).
$$
$$
\mu_0^{-1}(\xi)={1\over{\sqrt{1-4|\xi|^2}}}=|\xi|^{-1}(-{{\sqrt{-1}}\over2}+O(|\xi|^{-2})).
$$
This implies that
$$\begin{array}{ll}
e^{\mu_{\pm}(\xi)t}&=e^{-t/2}e^{\pm
\sqrt{-1}|\xi|t}(1+(\sum\limits_{j=1}\limits^{m-1}(\pm
a_j)|\xi|^{1-2j})t+\cdots\\&+{1\over{m!}}(\sum\limits_{j=1}\limits^{m-1}(\pm
a_j)|\xi|^{1-2j})^mt^m+R^{\pm}(\xi,t)),\end{array}
$$
where $R^{\pm}(\xi,t)\leq (1+t)^{m+1}(1+|\xi|)^{1-2m}.$\\\\
Denote
$$
\hat{w}(\xi,t)=(2\pi)^{-n/2}\sin(|\xi|t)/{|\xi|},~~\hat{w}_t=(2\pi)^{-n/2}\cos(|\xi|t).
$$
Since $$
 \partial_t^h\hat{G}^+(\xi,t)={{(\mu_+(\xi))^h}\over{\mu_0}}e^{\mu_+(\xi)t},~~
\partial_t^h\hat{G}^-(\xi,t)=-{{(\mu_-(\xi))^h}\over{\mu_0}}e^{\mu_-(\xi)t}.
$$
By a direct and a little tedious calculation we get that,
$$
\begin{array}{ll}\partial_t^h\hat{G}_3(\xi,t)&=e^{-t/2}\hat{w}_t(\sum\limits_{j=0}
\limits^{h-1}p^1_{1j}(t)q^1_{1j}(\xi
)+\sum\limits_{j=1}\limits^{2m-2}p^1_{2j}(t)q^1_{2j}(\xi)+\hat{R}^1(\xi,t))\\&\quad
\ \ +
e^{-t/2}\hat{w}(\sum\limits_{j=0}\limits^{h}p^2_{1j}(t)q^2_{1j}(\xi
)+\sum\limits_{j=1}\limits^{2m-2}p^2_{2j}(t)q^2_{2j}(\xi)+\hat{R}^2(\xi,t)),\end{array}
$$
here
$$
\begin{array}{ll} &p_{1j}^1(t)\leq
C(1+t)^{h-1-j},~~q^1_{1j}(\xi)=\chi_3(\xi)|\xi|^j,~~0\leq j\leq
h-1;\\&\\&p_{2j}^1(t)\leq
C(1+t)^{h+j},~~q^1_{2j}(\xi)=\chi_3(\xi)|\xi|^{-j},~~1\leq j\leq
2m-2;\\&\\&p_{1j}^2(t)\leq
C(1+t)^{h-j},~~q^2_{1j}(\xi)=\chi_3(\xi)|\xi|^j,~~0\leq j\leq
h;\\&\\&p_{2j}^2(t)\leq
C(1+t)^{h+j},~~q^2_{2j}(\xi)=\chi_3(\xi)|\xi|^{-j},~~1\leq j\leq
2m-2;\\&\\&|\hat{R}^1(\xi,t)|,~~|\hat{R}^2(\xi,t)|\leq
C(1+t)^{m+1}(1+|\xi|)^{h+1-2m}.
  \end{array}
$$
In the following we denote $q^1_{1j}(D_x)(0\leq j\leq
h-1),~q^2_{1j}(D_x)(0\leq j\leq h),~q^1_{2j}(D_x)(0\leq j\leq
2m-2),~q^2_{2j}(D_x)(0\leq j\leq 2m-2),$
$w(D_x,t),~w_t(D_x,t),\\~R^1(D_x,t),~R^2(D_x,t)$ the
pseudo-differential operators with symbols $q^1_{1j}(\xi)(0\leq
j\leq h-1),~q^2_{1j}(\xi)(0\leq j\leq h),~q^1_{2j}(\xi)(0\leq j\leq
2m-2),~q^2_{2j}(\xi)(0\leq j\leq
2m-2),~\hat{w}(\xi,t),~\hat{w}_t(\xi,t),~\hat{R}^1(\xi,t),~\hat{R}^2(\xi,t)$
respectively. It is easy to get that, for any multi-indexes
$\beta,~|\beta|\geq1$,
$$
\begin{array}{ll} &|D_{\xi}^{\beta}q^1_{2j}(\xi)|\leq
C|\xi|^{-1-|\beta|},~~|D_{\xi}^{\beta}q^2_{2j}(\xi)|\leq
C|\xi|^{-1-|\beta|},~~1\leq j\leq 2m-2;\\&\\&
|D_{\xi}^{\beta}\chi_3(\xi)|\leq
C|\xi|^{-1-|\beta|},~~|D_{\xi}^{\beta}(|\xi|^{-1}\chi_3(\xi))|\leq
C|\xi|^{-1-|\beta|};\end{array}$$
$$\begin{array}{ll}&{\rm supp}~
q^1_{1j}(0\leq j\leq h-1),~{\rm supp}~ q^2_{1j}(0\leq j\leq
h)\subset O_{R-1}=\{\xi;~|\xi|>R-1\};
\\&\\&{\rm supp}~ q^1_{2j}(1\leq j\leq 2m-2),~{\rm
supp}~ q^2_{2j}(1\leq j\leq 2m-2)\subset O_{R-1}.
 \end{array}
$$
By Lemma \ref{23} we have the following lemma.

 \bl\label{28}
For $R$ being sufficiently large, there exist distributions
$\bar{q}^1_{1j}(x),\\~\tilde{q}^1_{1j}(x),~0\leq j\leq h-1;$
$\bar{q}^2_{1j}(x),~\tilde{q}^2_{1j}(x),~0\leq j\leq h;$ $
\bar{q}^1_{2j}(x),~\tilde{q}^1_{2j}(x),~1\leq j\leq 2m-2;$ $
\bar{q}^2_{2j}(x),~\tilde{q}^2_{2j}(x),~1\leq j\leq 2m-2$ and
constant $C_0$ such that
$$
\begin{array}{ll}
&q^1_{1j}(D_x)\delta(x)=(-\triangle)^{{{j}\over2}}(\bar{q}^1_{1j}+\tilde{q}^1_{1j}+C_0\delta(x)),~0\leq
j\leq
h-1;\\&\\&q^2_{1j}(D_x)\delta(x)=(-\triangle)^{{{j}\over2}}(\bar{q}^2_{1j}+\tilde{q}^2_{1j}+C_0\delta(x)),~0\leq
j\leq
h;\\&\\&q^1_{2j}(D_x)\delta(x)=\bar{q}^1_{2j}+\tilde{q}^1_{2j}+C_0\delta(x),~1\leq
j\leq
2m-2;\\&\\&q^2_{2j}(D_x)\delta(x)=\bar{q}^2_{2j}+\tilde{q}^2_{2j}+C_0\delta(x),~1\leq
j\leq 2m-2;
 \end{array}
$$
and $$ \begin{array}{ll} &|D_x^{\alpha}\bar{q}^1_{1j}|(0\leq j\leq
h-1),~|D_x^{\alpha}\bar{q}^2_{1j}|(0\leq j\leq h)\leq
C(1+|x|^2)^{-N};\\&\\&|D_x^{\alpha}\bar{q}^1_{2j}|(1\leq j\leq
2m-2),~|D_x^{\alpha}\bar{q}^2_{2j}|(1\leq j\leq 2m-2)\leq
C(1+|x|^2)^{-N};\\&\\& \|\tilde{q}^1_{1j}\|_{L^1}(0\leq j\leq
h-1),~\|\tilde{q}^2_{1j}\|_{L^1}(0\leq j\leq h)\leq C;\\&\\&
\|\tilde{q}^1_{2j}\|_{L^1}(1\leq j\leq
2m-2),~\|\tilde{q}^2_{2j}\|_{L^1}(1\leq j\leq 2m-2)\leq C;
\\&\\&{\rm supp}~ \tilde{q}^1_{1j}(0\leq j\leq h-1),~{\rm
supp}~ \tilde{q}^2_{1j}(0\leq j\leq h)\subset
\{x;~|x|<2\varepsilon_1\};
\\&\\&{\rm supp}~ \tilde{q}^1_{2j}(1\leq j\leq 2m-2),~{\rm
supp}~ \tilde{q}^2_{2j}(1\leq j\leq 2m-2)\subset
\{x;~|x|<2\varepsilon_1\},
 \end{array}
$$
with $\varepsilon_1$ being sufficiently small.
 \el

Let
$$\begin{array}{ll}
&Q^1_{1j}(x)=\tilde{q}^1_{1j}(x)+C_0\delta(x),~0\leq j\leq
h-1;\\&\\&Q^2_{1j}(x)=\tilde{q}^2_{1j}(x)+C_0\delta(x),~0\leq j\leq
h;\end{array}$$
$$\begin{array}{ll}&Q^1_{2j}(x)=\tilde{q}^1_{2j}(x)+C_0\delta(x),~1\leq
j\leq 2m-2;\\&\\&Q^2_{2j}(x)=\tilde{q}^2_{2j}(x)+C_0\delta(x),~1\leq
j\leq 2m-2,
 \end{array}
$$
and $$ \begin{array}{ll} &
L^1_{1j}(x,t)=p^1_{1j}(t)w_t(D_x,t)(-\triangle)^{{{j}\over2}}Q^1_{1j}(x),~0\leq
j\leq
h-1;\\&\\&L^2_{1j}(x,t)=p^2_{1j}(t)w(D_x,t)(-\triangle)^{{{j}\over2}}Q^2_{1j}(x),~0\leq
j\leq
h;\\&\\&L_{2j}(x,t)=p^1_{2j}(t)w_t(D_x,t)Q^1_{2j}(x)+p^2_{2j}(t)w(D_x,t)Q^2_{2j}(x),~1\leq
j\leq 2m-2,
 \end{array}
$$
we have the following proposition.

 \bp\label{29} For $R$
sufficiently large, there exists distribution
$$
K^h_m(x,t)=e^{-t/2}(\sum\limits_{j=0}\limits^{h-1}L^1_{1j}(x,t)+\sum\limits_{j=0}\limits^{h}L^2_{1j}(x,t)
+\sum\limits_{j=1}\limits^{2m-2}L_{2j}(x,t))
$$ such that for $m\geq[{{|\alpha|+n+h+3}\over2}],$ we have that
$$
|D_x^{\alpha}(\partial_t^hG_3-K^h_m)(x,t)|\leq Ce^{-t/4}B_N(|x|,t).
$$
 \ep

The proof of Proposition \ref{29} can be seen in \cite{LW}.

 By Proposition \ref{21}, Proposition \ref{22} and Proposition \ref{29},
 we have the following proposition on the Green function.

 \bp\label{210}
For any integer $h\geq0$, any multi-index $\alpha$, and
$m\geq[{{|\alpha|+n+h+3}\over2}],$ we have that
$$
|D_x^{\alpha}(\partial_t^hG-K^h_m)(x,t)|\leq
C(1+t)^{-(n+|\alpha|+2h)/2}B_N(|x|,t),
$$
where $N>n$ can be big enough.
 \ep

\section{Pointwise estimates}

In this section, we aim at verifying that the solution obtained in
Theorem \ref{ge} satisfies the pointwise decay estimates expressed in Theorem \ref{pe}.

 By Duhamel's principle, the solution to \eqref{1a}\eqref{IC} can be expressed as following,
$$\begin{array}{ll}
u(x,t)=&G(x-\cdot,t)\ast (u_0+u_1)(\cdot)+\partial_t
G(x-\cdot,t)\ast
u_0(\cdot)\\&\\&-\int^t_0G(x-\cdot,t-\tau)\ast(|u|^{\theta}u)(\cdot,\tau)d\tau.\end{array}
$$

 We denote the solution
to the corresponding linear dissipative wave equation as $\bar{u}$,
then $$ \bar{u}(x,t):=G(x-\cdot,t)\ast (u_0+u_1)(\cdot)+\partial_t
G(x-\cdot,t)\ast u_0(\cdot).
$$
 Denote $$
\tilde{u}(x,t):=\int^t_0G(x-\cdot,t-\tau)\ast(|u|^{\theta}u)(\cdot,\tau)d\tau,
 $$ then the solution $u$ to (\ref{1a}) can be expressed as:
 $u=\bar{u}-\tilde{u}$.

In \cite{LW}, the following pointwise estimate of the solution $\bar {u}$ to the linear problem is obtained. 
\begin{thm}$^{\cite{LW}}$\label{lr}
 Assume that $(u_0, u_1)\in H^{s+1}\times H^{s}, s>n$
 is an integer, and for any multi-index
$\alpha\in \mathbb{Z}^n,~|\alpha|< s-{n\over2}$, there exists
$r>{n\over2}$ such that
$|D^{\alpha}_xu_0(x)|+|D^{\alpha}_xu_1(x)|\leq C(1+|x|^2)^{-r},$
then the solution $\bar{u}$ satisfies, for
$|\alpha|+h<s-n,$
\begin{equation}\label{lr1}
|\partial_t^hD_x^{\alpha}\bar{u}(x,t)|\leq
C(1+t)^{-(n+|\alpha|+2h)/2}(1+{{|x|^2}\over{1+t}})^{-r}. \end{equation}
\end{thm}

Now we give some lemmas which will be used later.

\bl\label{41}
Assume $n\geq1$, then the following inequalities hold,

(1). If $\tau\in[0, t]$, and $A^2\geq t$, then
$$
(1+{{A^2}\over{1+\tau}})^{-n}\leq
2^n({{1+\tau}\over{1+t}})^n(1+{{A^2}\over{1+t}})^{-n}.
$$

(2). If $A^2\leq t$, then $1\leq 2^n(1+{{A^2}\over{1+t}})^{-n}$.

 \el

 \bl\label{42}
Assume that $0\leq\tau\leq t$ and $h(x,\tau)$ satisfies $$
D^{\alpha}_xh(x, \tau)\leq C(1+\tau)^{-{{\theta
n+{|\alpha|}}\over2}}(1+{{|x|^2}\over{1+\tau}})^{-r},
$$ then
we have that,
$$\begin{array}{ll}
(1).& \int_{|z|=1}|D_x^{\alpha}h(x+tz, \tau)|dS_z\leq
C(1+\tau)^{-{{\theta
n+{|\alpha|}}\over2}}(1+t)^{2r}(1+{{|x|^2}\over{1+\tau}})^{-r}.
\\&\\(2).& \int_{|z|\leq1}{{|D_x^{\alpha}h(x+tz,
\tau)|}\over{\sqrt{1-|z|^2}}}dV_z\leq C(1+\tau)^{-{{\theta
n+{|\alpha|-1}}\over2}}(1+t)^{2r}(1+{{|x|^2}\over{1+\tau}})^{-r}.\end{array}
$$
 \el
 \begin{proof}  (1). By using Lemma $(\ref{24})_1$,
$$\begin{array}{ll}\int_{|z|=1}|D_x^{\alpha}h(x+tz,
\tau)|dS_z&\leq C\int^t_0(1+\tau)^{-{{\theta
n+{|\alpha|}}\over2}}(1+{{|x+tz|^2}\over{1+\tau}})^{-r}dS_z
\\&\leq C(1+\tau)^{-{{\theta
n+{|\alpha|}}\over2}}(1+t)^{2r}(1+{{|x|^2}\over{1+\tau}})^{-r}.\end{array}
$$
(2). By H$\ddot{o}$lder inequality and Lemma $(\ref{24})_2$,
 $$\begin{array}{ll}&\int_{|z|\leq1}{{|D_x^{\alpha}h(x+tz,
\tau)|}\over{\sqrt{1-|z|^2}}}dV_z\\&\leq
(\int_{|z|\leq1}|D_x^{\alpha}h(x+tz,
\tau)|^3dV_z)^{1\over3}(\int_{|z|\leq1}({1\over{\sqrt{1-|z|^2}}})^{3\over2}dV_z)^{2\over3}
\\&\leq C(1+\tau)^{-{{\theta
n+{|\alpha|-1}}\over2}}(\int_{|z|\leq1}{{(1+{{|x+tz|^2}\over{1+\tau}})^{-3r}}
\over{\sqrt{1+\tau}}}dV_z)^{1\over3}(\int^1_0(1-r^2)^{-{3\over4}}r^{n-1}dr)^{{2\over3}}
\\&\leq C(1+\tau)^{-{{\theta
n+{|\alpha|-1}}\over2}}(1+t)^{2r}(1+{{|x|^2}\over{1+\tau}})^{-r}.\end{array}
$$
 Thus we complete the proof of Lemma \ref{42}. 
\end{proof}
\begin{proof}[Proof of Theorem \ref{pe}]
 For $s>n$, denote 
$$
\begin{array}{ll}
&\varphi_{\alpha}(x,t):=(1+t)^{{n+|\alpha|}\over2}(B_r(|x|,
t))^{-1},\ \ r>{n\over2},\\&\\
  &M(T):=\sup\limits_{\begin{array}{ll}&(x,\tau)\in
\mathbb{R}^n\times
[0,T)\\&|\alpha|+h<s-n\end{array}}{|D_x^{\alpha}\partial_{\tau}^hu(x,\tau)|\varphi_{\alpha}(x,\tau)}.
\end{array}
$$
Now we come to make estimates to $\tilde{u}(x,t)$ under the assumption that $s>n$ and $\theta\geq2+[{1\over n}]$.

By induction argument, we obtain the following expression,
\begin{equation}\label{J}
\begin{array}{ll}
\p_t^h\tilde{u}(x,t)
&=\p_t^h\int^t_0G(t-\tau)\ast(|u|^{\theta}u)(\tau)d\tau\\
&=\sum\limits^{(h-1)_+}_{j=0}\p_t^jG(t)\ast
\p_t^{(h-1)_+-j}(|u|^{\theta}u)(0)+\int^t_0G(t-\tau)\ast\p_{\tau}^h(|u|^{\theta}u)(\tau)d\tau\\
&=: J_1+J_2,
\end{array}
\end{equation}
where $(h-1)_+=\max{\{h-1,0\}}$.

From \eqref{3h} and \eqref{3i}, we know that $\p_t^{(h-1)_+-j}(|u|^{\theta}u)(x,0)$ is
a polynomial with arguments $\Delta^iu_0(x)$ and $\Delta^ku_1(x)$, $0\leq i\leq [{(h-1)_+-j\over2}], \ 0\leq k\leq [{(h-1)_+-j-1\over2}].$
By using the similar estimates as that for $\bar{u}$ in
\cite{LW}, we obtain the following estimate for $D_x^{\alpha}J_1$,
 $$
|D_x^{\alpha}J_1|\leq CE_0(1+t)^{-{n+|\alpha|\over2}}
B_N(|x|,t).
$$
As for the estimate to $D_x^{\alpha}J_2$, we divide it
  as following ,
$$
\begin{array}{ll}
D_x^{\alpha}J_2
&=\int_0^{t\over2}\int_{\{y: |x|\leq2|y|\}}D_x^{\alpha}(G-K_m^0)(x-y,t-\tau)\partial_{\tau}^h
(|u|^{\theta}u)(y,\tau)dyd\tau
\\\\
&\quad +\int_0^{t\over2}\int_{\{y: |x|\geq2|y|\}}D_x^{\alpha}(G-K_m^0)(x-y,t-\tau)\p_{\tau}^h
(|u|^{\theta}u)(y,\tau)dyd\tau
\\\\
&\quad +\int^t_{t\over2}\int_{\{y: |x|\leq2|y|\}}(G-K_m^0)(x-y,t-\tau)
D_y^{\alpha}\p_{\tau}^h(|u|^{\theta}u)(y,\tau)dyd\tau
\\\\
&\quad +\int^t_{t\over2}\int_{\{y: |x|\geq2|y|\}}(G-K_m^0)(x-y,t-\tau)
D_y^{\alpha}\p_{\tau}^h
(|u|^{\theta}u)(y,\tau)dyd\tau
\\\\
&\quad +\int^{t}_0\int_{\mathbb{R}^n}K_m^0(x-y,t-\tau)D_y^{\alpha}\p_{\tau}^h
(|u|^{\theta}u)(y,\tau)dyd\tau\\\\
&=:J_{21}+J_{22}+J_{23}+J_{24}+J_{25}.
\end{array}
$$ 

 Next we estimate
$J_{2i}(i=1,2,3,4,5)$ respectively by using Proposition \ref{210} (with $N\geq r$) and  the fact that $B(|x|,t)\leq1$ and is an increasing function of $t$ and decreasing function of $|x|$ . 

By the definition of $M(T)$, we have that
$$
\begin{array}{ll}
|J_{21}|&\leq C \int_0^{t\over2}\int_{\{y: |x|\leq2|y|\}}
(1+t-\tau)^{-{{n+|\alpha|}\over2}}B_N(|x-y|,t-\tau)
\\&\qquad
M(T)^{\theta+1}(1+\tau)^{-{{(\theta+1)n+2h}\over2}}
B_r(|y|,\tau)dyd\tau.
\end{array}
$$
Now we estimate $J_{21}$ in two cases.

Case 1. $|x|^2\geq t$.  We have
$$
\begin{array}{ll}
|J_{21}|&\leq CM(T)^{\theta+1}\int_0^{t\over2}\int_{\{y: |x|\leq2|y|\}}
(1+t-\tau)^{-{{n+|\alpha|}\over2}}
\\&\qquad
B_N(|x-y|,t-\tau)(1+\tau)^{-{{(\theta+1)n}\over2}}
B_r(|x|,\tau)dyd\tau
\\&\leq
CM(T)^{\theta+1}B_r(|x|,t)\int_0^{t\over2}\int_{\{y: |x|\leq2|y|\}}
(1+t-\tau)^{-{{n+|\alpha|}\over2}}
\\&\qquad
B_N(|x-y|,t-\tau)(1+\tau)^{-{{(\theta+1)n}\over2}}({{1+\tau}\over{1+t}})^{r}
dyd\tau
\\&\leq CM(T)^{\theta+1}B_r(|x|,t)\int_0^{t\over2}
(1+t-\tau)^{-{{|\alpha|}\over2}}(1+\tau)^{-{{(\theta+1)
n}\over2}}({{1+\tau}\over{1+t}})^{n\over2} d\tau
\\&\leq
CM(T)^{\theta+1}B_r(|x|,t)(1+t)^{-{{ n+|\alpha|}\over2}},
\end{array}
$$
here in the second inequality we used  Lemma $\ref{41}$ (1).

Case 2. $|x|^2\leq t$.  We have
$$
\begin{array}{ll}
|J_{21}|&\leq  CM(T)^{\theta+1}\int_0^{t\over2}\int_{\{y: |x|\leq2|y|\}}
(1+t-\tau)^{-{{n+|\alpha|}\over2}}
\\&\qquad(1+\tau)^{-{{(\theta+1)n+2h}\over2}}
B_r(|y|,\tau)dyd\tau
\\
&\leq CM(T)^{\theta+1}\int_0^{t\over2}
(1+t-\tau)^{-{{n+|\alpha|}\over2}}(1+\tau)^{-{{\theta n+2h}\over2}}d\tau
\\&\leq CM(T)^{\theta+1}(1+t)^{-{{
n+|\alpha|}\over2}}
\\&\leq CM(T)^{\theta+1}B_r(|x|,t)(1+t)^{-{{
n+|\alpha|}\over2}},
\end{array}
$$
here in the last inequality we used  Lemma $\ref{41}$ (2).

Combining the two cases, we have that
$$
|J_{21}|\leq CM(T)^{\theta+1}(\varphi_{\alpha}(x,t))^{-1}.
$$
For $J_{22}$, we have
$$
\begin{array}{ll}
|J_{22}|&\leq C \int_0^{t\over2}\int_{\{y: |x|\geq2|y|\}}
(1+t-\tau)^{-{{n+|\alpha|}\over2}}B_N(|x-y|,t-\tau)
\\&\qquad
M(T)^{\theta+1}(1+\tau)^{-{{(\theta+1)n+2h}\over2}}
B_r(|y|,\tau)dyd\tau, 
\end{array}
$$
noticing that if $|x|\geq2|y|$, then $|x-y|\geq {|x|\over2}$, it yields that
$$
\begin{array}{ll}
|J_{22}| &\leq
CM(T)^{\theta+1}\int_0^{t\over2}\int_{\{y: |x|\geq2|y|\}}(1+t-\tau)^{-{{n+|\alpha|}\over2}}
\\&\qquad
B_N(|x|,t)(1+\tau)^{-{{(\theta+1)n+2h}\over2}}
B_r(|y|,\tau)dyd\tau
\\&\leq
CM(T)^{\theta+1}B_N(|x|,t)\int_0^{t\over2}
(1+t-\tau)^{-{{n+|\alpha|}\over2}}(1+\tau)^{-{{\theta n+2h}\over2}}d\tau
\\&\leq CM(T)^{\theta+1}B_N(|x|,t)(1+t)^{-{{
n+|\alpha|}\over2}}
\\&\leq CM(T)^{\theta+1}(\varphi_{\alpha}(x,t))^{-1}.
\end{array}
$$
For $J_{23}$, by using the monotonic properties of $B(|x|,t)$ with respect to $|x|$ and $t$, we have that
$$
\begin{array}{ll}
|J_{23}|&\leq C \int^t_{t\over2}\int_{\{y: |x|\leq2|y|\}}
(1+t-\tau)^{-{{n}\over2}}B_N(|x-y|,t-\tau)
\\&\qquad
M(T)^{\theta+1}(1+\tau)^{-{{(\theta+1)n+|\alpha|+2h}\over2}}B_r(|y|,\tau)dyd\tau
\\&\leq CM(T)^{\theta+1}\int^t_{t\over2}\int_{\{y: |x|\leq2|y|\}}
(1+t-\tau)^{-{{n}\over2}}
\\&\qquad
B_N(|x-y|,t-\tau)(1+\tau)^{-{{(\theta+1)n+|\alpha|+2h}\over2}}
B_r(|x|,t)dyd\tau
\\&\leq CM(T)^{\theta+1}B_r(|x|,t)\int^t_{t\over2}
(1+\tau)^{-{{(\theta+1) n+|\alpha|+2h}\over2}} d\tau
\\&\leq
CM(T)^{\theta+1}B_r(|x|,t)(1+t)^{-{{
n+|\alpha|}\over2}}
\\&\leq CM(T)^{\theta+1}(\varphi_{\alpha,
h}(x,t))^{-1}.
\end{array}
$$
For $J_{24}$,  similar to $J_{22}$ we have that
$$
\begin{array}{ll}
|J_{24}|&\leq C \int^t_{t\over2}\int_{\{y: |x|\geq2|y|\}}
(1+t-\tau)^{-{{n}\over2}}B_N(|x-y|,t-\tau)
\\&\qquad
M(T)^{\theta+1}(1+\tau)^{-{{(\theta+1)n+|\alpha|+2h}\over2}}B_r(|y|,\tau)dyd\tau
\\&\leq CM(T)^{\theta+1}\int^t_{t\over2}\int_{\{y: |x|\geq2|y|\}}
(1+t-\tau)^{-{{n}\over2}}
\\&\qquad
B_N(|x|,t)(1+\tau)^{-{{(\theta+1)n+|\alpha|+2h}\over2}}
B_r(|y|,\tau)dyd\tau
\\&\leq
CM(T)^{\theta+1}B_N(|x|,t)\int^t_{t\over2}
(1+t-\tau)^{-{{n}\over2}}(1+\tau)^{-{{\theta n+|\alpha|+2h}\over2}}
d\tau
\\&\leq CM(T)^{\theta+1}B_N(|x|,t)(1+t)^{-{{
n+|\alpha|+2h}\over2}}
\\&\leq CM(T)^{\theta+1} (\varphi_{\alpha}(x,t))^{-1}.
\end{array}
$$
Finally we come to estimate $J_{25}.$
From the definition of $K_m^0$ we have
\begin{equation}\label{j25}
\begin{array}{ll}
|J_{25}|&=|\int^t_0\int_{\mathbb{R}^n}e^{-{{t-\tau}\over2}}
p^2_{10}(t-\tau)w_t(D_x,t-\tau)(\tilde{q}^2_{10}+C_0\delta)(x-y)
\\&\quad+
\sum\limits^{2m-2}\limits_{j=1}[p^1_{2j}(t-\tau)w_t(D_x,t-\tau)(\tilde{q}^1_{2j}+C_0\delta)(x-y)
\\&\quad
 + p^2_{2j}(t-\tau)w(D_x,t-\tau)(\tilde{q}^2_{2j}+C_0\delta)(x-y)]
 \}D_y^{\alpha}\p_{\tau}^h(|u|^{\theta}u)(y,\tau)dyd\tau|
\\&
 \leq
|\int^t_0\int_{\mathbb{R}^n}e^{-{{t-\tau}\over2}
}p^2_{10}(t-\tau)w(D_x,t-\tau)\tilde{q}^2_{10}(x-y)
\\&\quad +
\sum\limits^{2m-2}\limits_{j=1}[p^1_{2j}(t-\tau)w_t(D_x,t-\tau)\tilde{q}^1_{2j}(x-y)
\\&\quad
 + p^2_{2j}(t-\tau)w(D_x,t-\tau)\tilde{q}^2_{2j}(x-y)]
 \}D_y^{\alpha}(|u|^{\theta}u)(y,\tau)dyd\tau|
\\&\quad+
|\int^t_{t\over2}\int_{\mathbb{R}^n}e^{-{{t-\tau}\over2
}}\{\sum\limits^{h-1}p^2_{10}(t-\tau)w(D_x,t-\tau)
\\&\quad+
\sum\limits^{2m-2}\limits_{j=1}[p^1_{2j}(t-\tau)w_t(D_x,t-\tau) + p^2_{2j}(t-\tau)w(D_x,t-\tau)] \}
\\&\qquad C_0\delta(x-y)D_y^{\alpha}(|u|^{\theta}u)(y,\tau)dyd\tau|
\\&=: K_1+K_2.
\end{array}
\end{equation}
 
By using Lemma \ref{25}, Lemma \ref{28}, Lemma \ref{42} and the fact that   $(1+{{|y|^2}\over{1+t}})^{-1}\leq C(1+{{|x|^2}\over{1+t}})^{-1},$ if
 $|x-y|\leq2\varepsilon_1$, $K_1$ can be estimated as follows,
\begin{equation}\label{k1}
\begin{array}{ll}
K_1&\leq
 C\int^t_0\int_{\mathbb{R}^n}e^{-
{{t-\tau}\over4}}
(\tilde{q}^2_{10}|+\sum\limits^{2m-2}\limits_{j=1}(|\tilde{q}^1_{2j}|+|\tilde{q}^2_{2j}|))(x-y)
\\&\qquad
M(T)^{\theta+1}(1+\tau)^{-{{(\theta+1)
n+|\alpha|}\over2}}(1+{{|y|^2}\over{1+\tau}})^{-r}dyd\tau
\\&\leq
CM(T)^{\theta+1}(1+t)^{-{{
n+|\alpha|}\over2}}(1+{{|x|^2}\over{1+t}})^{-r}.
\end{array}
\end{equation}
By using Lemma \ref{25} and Lemma \ref{42}, $K_2$ can be controlled by
\begin{equation}\label{k2}
\begin{array}{ll}
K_2&\leq C\int^t_0e^{-{{t-\tau}\over4}}
M(T)^{\theta+1}(1+\tau)^{-{{(\theta+1) n+|\alpha|
}\over2}}(1+{{|x|^2}\over{1+\tau}})^{-r}d\tau
\\&\leq
CM(T)^{\theta+1}(1+t)^{-{{
n+|\alpha|}\over2}}(1+{{|x|^2}\over{1+t}})^{-r}.
\end{array}
\end{equation}
\eqref{j25}, \eqref{k1} and \eqref{k2}  yield that 
$$
|J_{25}|\leq CM(T)^{\theta+1}(\varphi_{\alpha}(x,t))^{-1}.
$$
 Combining the estimates for $J_{2i}, i =1,2,3,4,5,$ we have that
\begin{equation}\label{pe2}
|D_x^{\alpha}\partial_t^h\tilde{u}(x,t)|\leq
CM(T)^{\theta+1}(\varphi_{\alpha}(x,t))^{-1}.
\end{equation}
\eqref{pe2} combined with \eqref{lr1} Theorem \ref{lr} yields that,
$$
M(T)\leq C(E_0+M(T)^{\theta+1}).
$$
 Since $\theta\geq2$,  we have that $M(T)\leq CE_0$ if $E_0$ is suitably small.  It yields that
$$
|D_x^{\alpha}\partial_t^hu(x,t)|\leq
CE_0(1+t)^{-{{n+|\alpha|}\over2}}B_r(|x|,t).
$$
Thus Theorem \ref{pe} is proved.
\end{proof}

\section*{Acknowledgement} This work was partially supported by Grant-in-Aid for JSPS Fellows. Also, the author would give thanks to professor  Weike Wang for the helpful discussion.

\end{document}